\documentclass[a4paper,11pt]{amsart}
\usepackage{amsmath,amssymb,amsbsy,bbm,mathtools,nicefrac}
\usepackage{geometry}
\usepackage{tikz}
\usetikzlibrary{arrows.meta,calc,positioning}
\usepackage[colorlinks=true,allcolors=black]{hyperref}
\DeclareMathAlphabet\mathbfcal{OMS}{cmsy}{b}{n}
\usepackage{enumitem}

\newcommand{\bE}{\mathbb{E}}
\renewcommand{\Pr}{\mathbb{P}}
\newcommand{\bR}{\mathbb{R}}

\newcommand{\bF}{\mathbb{F}}

\newcommand{\cF}{\mathcal{F}}

\newcommand{\cB}{\mathcal{B}}
\newcommand{\cT}{\mathcal T}
\newcommand{\cD}{\mathcal D}
\newcommand{\cK}{\mathcal K}
\newcommand{\cR}{\mathcal R}
\newcommand{\cS}{\mathcal S}
\newcommand{\cV}{\mathcal V}

\newcommand{\cC}{\mathcal C}
\newcommand{\cP}{\mathcal P}
\newcommand{\cJ}{\mathcal J}

\newcommand{\mX}{\boldsymbol X}
\newcommand{\mG}{\boldsymbol G}

\newcommand{\mR}{\boldsymbol R}
\newcommand{\mP}{\boldsymbol P}
\newcommand{\mPi}{\boldsymbol\Pi}
\newcommand{\mGamma}{\boldsymbol\Gamma}

\newcommand{\U}[2][]{{{#1}\uparrow{#2}}}
\newcommand{\UD}[3][]{{{#1}\!\underset{#3}{\uparrow}\!{#2}}}
\newcommand{\taua}{\tau_{\U{a}}}

\newcommand{\D}[2][]{{{#1}\downarrow{#2}}}
\newcommand{\DU}[3][]{{{#1}\!\overset{#2}{\downarrow}\!{#3}}}

\newcommand{\taub}{\tau_{\D{b}}}

\newcommand{\rhob}{\rho_{\D{b}}}\renewcommand{\rhob}{\rho_b}
\newcommand{\rhov}{\rho_{\D{v}}}\renewcommand{\rhov}{\rho_v}

\newcommand{\va}{\boldsymbol a}

\newcommand{\cEll}{\mathcal L}

\newcommand{\ve}{\boldsymbol e}

\newcommand{\vc}{\boldsymbol c}
\newcommand{\vi}{\boldsymbol 1}
\newcommand{\indF}[1]{\mathbbm{1}\left\{#1\right\}} 
\newcommand{\vlambda}{\boldsymbol \lambda}
\newcommand{\vp}{\boldsymbol p}

\newcommand{\vz}{\boldsymbol z}
\newcommand{\vgamma}{\boldsymbol \gamma}

\newcommand{\mA}{\boldsymbol A}
\newcommand{\mB}{\boldsymbol B}
\newcommand{\mE}{\boldsymbol E}

\newcommand{\mH}{\boldsymbol H}

\newcommand{\mDelta}{\boldsymbol \Delta}
\newcommand{\mEps}{\mathbfcal{E}}
\newcommand{\mId}{\boldsymbol I}
\newcommand{\mQ}{\boldsymbol Q}
\newcommand{\mU}{\boldsymbol U}

\newcommand{\mZ}{{\boldsymbol Z}}
\newcommand{\ZU}[2]{\mZ(\U[#1\!]{#2})}

\newcommand{\ZOx}{\ZU 0 x}

\newcommand{\mC}{{\boldsymbol C}}
\newcommand{\mF}{{\boldsymbol F}}
\newcommand{\mK}{{\boldsymbol K}}

\newcommand{\Exp}{\text{Exp}}

\newcommand{\uX}{\underline X}

\newcommand{\upord}{\mathord{\upharpoonright}}

\theoremstyle{plain}
\newtheorem{theorem}{Theorem} [section]
\newtheorem{corollary}[theorem]{Corollary}
\newtheorem{lemma}[theorem]{Lemma}
\newtheorem{remark}[theorem]{Remark}
\newtheorem{proposition}[theorem]{Proposition}

\theoremstyle{definition}

\title[Exit problems for Markov-modulated AIMD processes]{Exit problems for additive-increase and multiplicative-decrease Markov-modulated processes}
\author{Bernardo D'Auria}
\address{Department of Mathematics “Tullio Levi- Civita”, University of Padua, Italy}
\email{dauria@math.unipd.it}
\author[Z. Palmowski]{Zbigniew Palmowski}
\address{Faculty of Pure and Applied Mathematics,
Wroc\l aw University of Science and Technology,
Wyb. Wyspia\'nskiego 27, 50-370 Wroc\l aw, Poland}
\email{zbigniew.palmowski@pwr.edu.pl}
\thanks{Zbigniew Palmowski was partially supported by the National Science Centre (Poland): grant 2023/51/B/ST1/01270. Bernardo D'Auria is a member of the Gruppo Nazionale Calcolo Scientifico-Istituto Nazionale di Alta Matematica (GNCS-INdAM) and acknowledges the financial support by the Italian SID project BIRD239937/23.}
\date{\today}

\begin{document}
\begin{abstract}
We study one-sided, two-sided and reflected exit problems for a finite-state Markov-modulated additive-increase and multiplicative-decrease process. 
For upward passage, a first-jump decomposition and a spatial Laplace transform lead to a recursion indexed by finite words of phases and to a convergent Green series. 
For downward passage, the first-jump operator is a strict contraction 
and the known contribution from $[p_i b,b]$ becomes an explicit boundary forcing term in the same spatial recursion.
These allow to derive the corresponding Green representation of the first passage matrix 
and its finite-dimensional compatibility condition.
The derivation of the two-sided exit identity is based on the Markov property, which in this context is equivalent to the cocycle property of the Laplace transform, together with the recursive equation at the upper level. 
For reflection at the running infimum, a finite $\bar p$-geometric method of steps realizes a certain Picard map that allows to find the exit matrix. 
As a result, we construct a Doob martingale related to the first passage time of the reflected Markov-modulated AIMD.
\end{abstract}
\maketitle

\section{Introduction}\label{sec:introduction}
Additive-increase and multiplicative-decrease (AIMD) processes form a particularly tractable class of piecewise deterministic Markov processes.
Between jump epochs the level increases linearly, while at a jump it is multiplied by a factor in $(0,1)$.
This simple path structure already produces nonlocal exit equations, because the state after a jump is a fixed fraction of the level reached immediately before the jump.
These processes can also be viewed as a basic model of partial stochastic resetting, and related mechanisms arise in search processes (see, e.g., \cite{Bel, Evans, White}), physics (see, e.g., \cite{18}), biology (see, e.g., \cite{4}), and queueing theory (see, e.g., \cite{MR2576022}); see also the survey~\cite{Gupta}.
In the context of processes with resetting, two key quantities of interest are the distributions of exit times from half-lines or intervals (see, e.g., \cite{Ralf}) and the stationary distribution (see, e.g., \cite{my}).
For scalar AIMD processes, one-sided and two-sided exit transforms can be expressed through scale-type functions and recursive identities; see~\cite{vanderhofstad2023}. Here we extend this framework to a finite-state Markov-modulated setting, where the interaction between the level process and the environmental phase gives rise to a richer structure.

We develop a unified approach to exit problems for a finite-state Markov-modulated AIMD process.
The modulation affects the process at the same epochs at which the multiplicative decreases occur.
In other words, if the phase immediately before a jump is $i$, then the level is multiplied by $p_i$ and the post-jump phase is $j$ with probability $q_{ij}$.
Between jumps, the level grows linearly at a rate determined by the environmental process~$J$.
The passage from the scalar to the Markov-modulated setting is not merely a matrix-valued reformulation.
Phase-dependent multiplicative factors produce different spatial dilations according to the successive phases visited by the process, and the resulting exit equations acquire a recursive structure indexed by finite phase words.

Exit transforms are naturally matrix-valued. Their $(i,j)$ entries record the initial phase $i$ and the phase $j$ at the relevant passage epoch. Matrix multiplication then becomes the exact algebraic expression of the strong Markov property when a trajectory is decomposed at an intermediate passage level and the intermediate phase is summed out.

The exit problems considered here are governed by two different renewal mechanisms, according to which information must be retained after a path decomposition. 

For ordinary upward passage the renewal structure is spatial: since the process has no upward jumps (\emph{upward skip-free}), every trajectory reaching a higher level must pass continuously through each intermediate level, and the strong Markov property therefore factorizes passage matrices across such levels. Combining this factorization with a first-jump decomposition couples the spatial Laplace transform at $s$ to its value at $s/p_i$. Each iteration selects the phase responsible for the next dilation, so successive iterations are naturally indexed by finite phase words. We identify a one-step spatial operator~$\cK$, whose resolvent $\cR=(I-\cK)^{-1}$ sums the contributions generated by all successive iterations of $\cK$, yielding the corresponding Green-series representation.

The same spatial resolvent also governs downward passage.
The lower boundary produces an additional term determined by the prescribed boundary values.
After this boundary term is separated from the unknown part, the remaining transform satisfies the same spatial recursion as in the upward problem.
The downward transform can therefore be represented through the same operator~$\cR$, while the remaining unknown is determined by a finite-dimensional compatibility condition.
The same framework also allows us to address the two-sided exit problem, which introduces additional difficulties due to the simultaneous presence of lower and upper boundaries.

The second renewal mechanism arises for reflection at the running infimum.
When a downward jump creates a new running minimum $y$, the reflected coordinate is reset to zero, but the absolute level $y$ and the new phase remain relevant for the future evolution. 
This leads to a restart formulation in which the running minimum is retained as part of the state. The resulting operator is a strict contraction, which yields existence and uniqueness of the continuation value. 
A finite geometric method of steps then gives a constructive representation of the reflected first-passage transform and the associated Doob martingale.

The paper is organized as follows. Section~\ref{sec:model} introduces the Markov-modulated AIMD process and the matrix convention. Section~\ref{sec:upward} derives the upward first-step equation and the common spatial resolvent and its phase-word expansion. Section~\ref{sec:downward} establishes the downward first-step contraction and derives the boundary forcing, the Green representation, and the finite-dimensional closure. Section~\ref{sec:two-sided} develops two-sided exit identities and a constructive continuation in the upper level. Section~\ref{sec:reflected} treats reflection at the running infimum through a strict-contraction restart equation and constructs the associated Doob martingale.

\section{Model and notation}\label{sec:model}
Let $E=\{1,\ldots,m\}$, and assume throughout that $c_i>0$, $\lambda_i>0$ and $p_i\in(0,1)$ for every~$i\in E$, and that $\mQ=(q_{ij})_{i,j\in E}$ is stochastic.
Set $\bar p=\max_{i\in E} p_i<1$, $\bar \lambda=\max_{i\in E} \lambda_i$ and $\bar c=\max_{i\in E} c_i$.
We consider the Markov process $\mX=(X,J)$ on $\bR_+\times E$.
Conditional on the pre-jump phase $i$, the level increases at velocity $1/c_i$ and the next jump occurs at rate $\lambda_i$.
At such a jump,
$$
(x,i)\longmapsto(p_i x,j)\qquad\hbox{with probability }q_{ij}.
$$
Thus, if $T_n$ denotes the $n$-th jump epoch, then $J$ is constant and $X$ is affine on $[T_n,T_{n+1})$, with slope $1/c_{J(T_n)}$.
At $T_{n+1}$ the multiplicative factor is determined by the pre-jump phase and the new phase is sampled according to $\mQ$; see Figure~\ref{fig:MMaimd}.
Since the jump rates are bounded on the finite phase space, only finitely many jumps occur on every compact time interval almost surely.

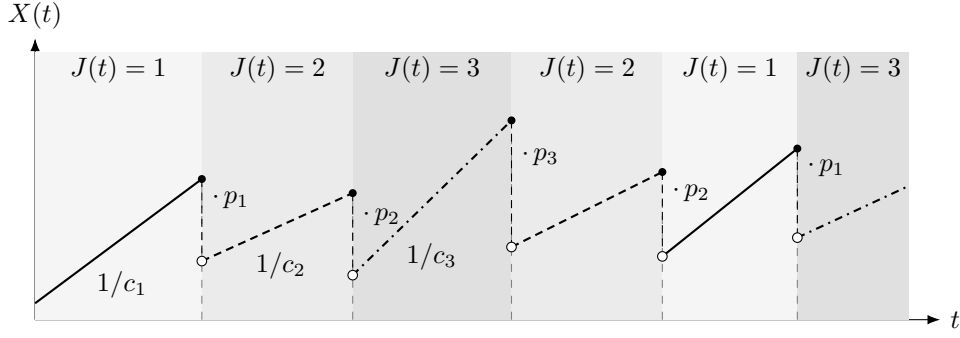
\begin{figure}[t]
\centering
\begin{tikzpicture}[x=1.05cm,y=0.62cm,>=Latex,font=\small]
\draw[->] (0,0) -- (11.4,0) node[right] {$t$};
\draw[->] (0,0) -- (0,6.0) node[above] {$X(t)$};
\foreach \xa/\xb/\lab in {0/2.1/1,2.1/4.0/2,4.0/6.0/3,6.0/7.9/2,7.9/9.6/1,9.6/11.0/3}{
  \ifnum\lab=1 \fill[gray!8] (\xa,0) rectangle (\xb,5.7); \fi
  \ifnum\lab=2 \fill[gray!16] (\xa,0) rectangle (\xb,5.7); \fi
  \ifnum\lab=3 \fill[gray!24] (\xa,0) rectangle (\xb,5.7); \fi
  \node at ({(\xa+\xb)/2},5.35) {$J(t)=\lab$};
}
\coordinate (A0) at (0,0.35); \coordinate (A1) at (2.1,3.0);
\coordinate (B1) at (2.1,1.25); \coordinate (B2) at (4.0,2.7);
\coordinate (C2) at (4.0,0.95); \coordinate (C3) at (6.0,4.25);
\coordinate (D3) at (6.0,1.55); \coordinate (D4) at (7.9,3.15);
\coordinate (E4) at (7.9,1.35); \coordinate (E5) at (9.6,3.65);
\coordinate (F5) at (9.6,1.75); \coordinate (F6) at (11.0,2.85);
\draw[black,thick] (A0)--(A1);
\draw[black,thick,densely dashed] (B1)--(B2);
\draw[black,thick,dash dot] (C2)--(C3);
\draw[black,thick,densely dashed] (D3)--(D4);
\draw[black,thick] (E4)--(E5);
\draw[black,thick,dash dot] (F5)--(F6);
\foreach \x/\yt/\yb/\pi in {2.1/3.0/1.25/p_1,4.0/2.7/0.95/p_2,6.0/4.25/1.55/p_3,7.9/3.15/1.35/p_2,9.6/3.65/1.75/p_1}{
  \draw[dashed,gray] (\x,0)--(\x,\yt);
  \draw[densely dashed] (\x,\yt)--(\x,\yb);
  \fill (\x,\yt) circle (1.5pt); \draw[fill=white] (\x,\yb) circle (1.8pt);
  \node[right] at (\x, {0.6*(\yt+\yb)}) {$\cdot\,\pi$};
}
\node[below right] at (0.65,1.25) {$1/c_1$};
\node[below right] at (2.65,1.65) {$1/c_2$};
\node[below right] at (4.55,1.85) {$1/c_3$};
\end{tikzpicture}
\caption{A sample path of the Markov-modulated AIMD process. 
While the process is in phase i, the level has slope $1/c_i$; 
at a jump it is multiplied by $p_i$, after which the new phase is sampled according to $\mQ$.}
\label{fig:MMaimd}
\end{figure}

If $N$ denotes the jump-counting process, the level coordinate may equivalently be written in differential form as
\begin{equation}\label{model.SDE}
dX(t)=\frac{1}{c_{J(t)}}dt+(p_{J(t-)}-1)X(t-)\,dN(t),
\end{equation}
with $N$ having predictable intensity $\lambda_{J(t-)}$.
We use the completed right-continuous natural filtration $\bF=(\cF_t)_{t\ge0}$ and write $\Pr_{x,i}$ and $\bE_{x,i}$ for probability and expectation conditional on $(X(0),J(0))=(x,i)$.
The vector $\ve_i$ denotes the $i$-th canonical row vector and $\vi$ the column vector of ones.
For a matrix $\mG=(G_{ij})$, we define the following submultiplicative norm
$$
|\mG|=\max_i\sum_j|G_{ij}|,
$$
and, fixed $b>0$, we define $\cB_b$ as the complete metric space of bounded matrix-valued functions $\mG$ on $[0,\infty)$ such that $\mG(x)=\mId$ for $x\le b$, endowed with the supremum row-sum norm
$$
\|\mG\|_\infty = \sup_{x\ge0}|\mG(x)| .
$$
We write $\mDelta(\va)$ for the diagonal matrix with diagonal vector $\va$, $\mEps(\va)=\mDelta(e^{\va})$ and $\mE_i=\Delta(\ve_i)$. Products and quotients of vectors are componentwise.
Matrix multiplication has the usual probabilistic interpretation.
Throughout $w>0$ and $z_i=c_i(\lambda_i+w)$, $\gamma_i=c_i\lambda_i/p_i$, so that
$$
\vz=\vc(\vlambda+w),\qquad \vgamma=\frac{\vc\vlambda}{\vp}.
$$
We set $\bar z=\max_{i\in E} z_i$ and $\bar \gamma=\max_{i\in E}\gamma_i$.
For $a\ge b\ge0$, define the first entrance times
$$
\taua=\inf\{t\ge0:X(t)\ge a\},\qquad \taub=\inf\{t\ge0:X(t)\le b\}.
$$
Since upward motion is continuous~(\emph{upward skip-free}), $X(\taua)=a$ on $\{\taua<\infty\}$. Downward passage, by contrast, occurs by a multiplicative jump and generally undershoots a fixed lower level $b$.


\section{The upward one-sided problem}\label{sec:upward}
We define one-sided discounted passage matrices by
\begin{equation}\label{Z}
\mZ_{ij}(w;\U[x]{a})=\bE_{x,i}[e^{-w\taua};J(\taua)=j],\qquad \mZ_{ij}(w;\D[x]{b})=\bE_{x,i}[e^{-w\taub};J(\taub)=j].
\end{equation}
These matrices are substochastic in the row-sum sense.

Our first main aim is find $\mZ_{ij}(w;\U[x]{a})=\bE_{x,i}[e^{-w\taua};J(\taua)=j]$. 
The key observation is that a passage from $x$ to $a$ can be decomposed at any intermediate level $y$, with a sum over the phase attained at $y$.
This absence of upward jumps giving a multiplicative passage identity is illustrated in Figure~\ref{fig:markov-factorization}
and formalised below.

\begin{theorem}[Upward passage]\label{thm:upward}
$\ZOx$ is invertible for every $x\ge0$.
Further, for $0\le x\le a$,
the one-sided discounted passage matrix satisfies the \emph{cocycle} identity
\begin{equation}\label{upward.cocycle.identity}
\mZ(\U[x]{a})=\mZ^{-1}(\U[0]{x})\mZ(\U[0]{a}).
\end{equation}
\end{theorem}

\begin{figure}[t]
\centering
\begin{tikzpicture}[x=0.95cm,y=0.68cm,>=Latex,font=\small]
\draw[->] (0,0)--(10.2,0) node[right] {$t$};
\draw[->] (0,0)--(0,4.8) node[above] {$X(t)$};
\draw[dashed,gray] (0,2.35)--(9.6,2.35); \node[left] at (0,2.35) {$x$};
\draw[dashed,gray] (0,4.05)--(9.6,4.05); \node[left] at (0,4.05) {$a$};
\draw[thick] (0,0.25)--(1.4,1.35)--(1.4,0.8)--(3.4,2.35)--(3.4,1.65)--(5.3,3.0)--(5.3,2.05)--(7.2,3.35)--(7.2,2.75)--(9.35,4.05);
\fill[black] (3.4,2.35) circle (2.2pt); \draw[black,dashed] (3.4,0)--(3.4,2.35);
\fill[black] (9.35,4.05) circle (2.2pt); \draw[black,dashed] (9.35,0)--(9.35,4.05);
\node[below] at (3.4,0) {$\tau_{\U{x}}$};
\node[below] at (9.35,0) {$\tau_{\U{a}}$};
\node at (1.9,2.75) {$\mZ(\U[0]{x})$};
\node at (7.25,3.78) {$\mZ(\U[x]{a})$};
\end{tikzpicture}
\caption{Strong-Markov factorization of the upward passage at an intermediate level $x$. The phase attained at $\tau_{\U{x}}$ is summed by matrix multiplication.}
\label{fig:markov-factorization}
\end{figure}
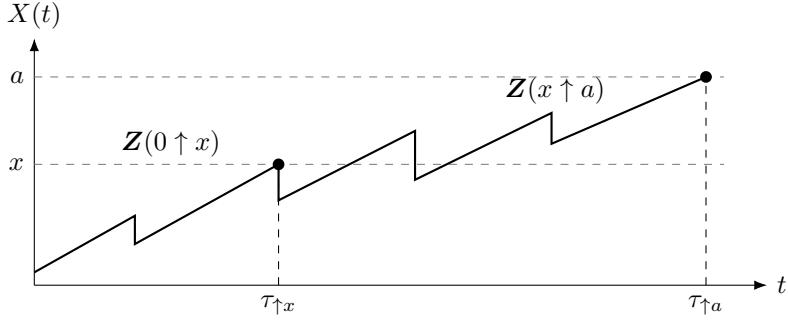

\begin{proof}
For $0<x<a$, the strong Markov property at $\tau_{\U{x}}$ gives, componentwise,
$$
\begin{aligned}
\mZ_{ij}(\U[0]{a})
&=\sum_{k=1}^m\bE_{0,i}[e^{-w\tau_{\U{x}}};J(\tau_{\U{x}})=k]\,
\bE_{x,k}[e^{-w\taua};J(\taua)=j]\\
&=\sum_{k=1}^m\mZ_{ik}(\U[0]{x})\mZ_{kj}(\U[x]{a}).
\end{aligned}
$$
Hence, $\mZ(\U[0]{a})=\ZOx\mZ(\U[x]{a})$ and \eqref{upward.cocycle.identity} follows.

To show that $\ZOx$ is invertible, note that the matrix valued function $a\rightarrow \mZ(\U[x]{a})$ is continuous and hence its determinant is continuous as well.
Moreover, for $h\downarrow0$, the probability of a jump before the deterministic climb from $x$ to $x+h$ is asymptotically equal to $\lambda h$ and we have $\mZ(\U[p_ix]{x+h})\leq 1$.
Therefore $\mZ(\U[x]{x+h})=\mId +{\rm O}(h)$ as $h\downarrow0$.
Thus the local passage matrices $\mZ(\U[x]{x+h})$ are all invertible for all $x\geq 0$ and sufficiently small level increment $h>0$.
A finite partition of $[0,nh]$ and the preceding cocycle identity gives that the matrix $\mZ(\U[0]{nh})=\prod_{k=0}^{n-1}\mZ(\U[kh]{(k+1)h})$ is invertible
which is equivalent that $\ZOx$ is invertible.
\end{proof}

To obtain a representation that can be iterated, we now condition on the first jump epoch $T$ when the process $(X,J)$  starts from $(0,i)$. If $T>c_i a$, the process reaches $a$ without a jump. If $T\le c_i a$, the pre-jump level is $T/c_i$ and the post-jump level is $p_iT/c_i$. Therefore
$$
\begin{aligned}
\mZ_{ij}(\U[0]{a})
={}&e^{-z_i a}\delta_{ij}
+\int_0^{c_i a}\lambda_i e^{-(w+\lambda_i)t}
\sum_{k=1}^m q_{ik}\mZ_{kj}(\U[p_it/c_i]{a})\,dt\\
={}&e^{-z_i a}\delta_{ij}
+\int_0^{p_i a}\gamma_i e^{-(z_i/p_i)y}
\sum_{k=1}^m q_{ik}\mZ_{kj}(\U[y]{a})\,dy.
\end{aligned}
$$
Using Theorem~\ref{thm:upward} in the integrand and setting $\mU(a)=\mZ^{-1}(\U[0]{a})$, we obtain
\begin{equation}\label{eq:upU}
\mId=\mEps(-a\vz)\mU(a)+\mDelta(\vgamma)
\int_0^\infty\mDelta(y\le\vp a)\mEps(-(\vz/\vp)y)\mQ\mU(y)\,dy.
\end{equation}
This identity is the matrix analogue of the scalar renewal equation for the reciprocal upward scale function, compare with~\cite[Equation~(2.4)]{vanderhofstad2023}.
We briefly justify the Laplace transform used below.
For $h>0$ small, uniformly in the starting level $x$, the probability of a jump before the deterministic climb from $x$ to $x+h$ is at most $1-e^{-\bar \lambda \bar c h}=O(h)$.
Hence, the local passage matrix $\mZ(\U[x]{x+h})$ differs from $\mId$ by $O(h)$ uniformly in $x$.
Choosing $h$ small enough, its inverse has norm at most $1+C_0h$, for some constant $C_0$.
By the cocycle property and a partition of $[0,a]$ into intervals of length at most $h$, there are constants $C_1,\kappa<\infty$ such that $|\mU(a)|\le C_1e^{\kappa a}$, $a\ge0$.
Thus, the following Laplace transform is well defined for every real~$s>\kappa$:
$$
\widetilde\mU(s)=\int_0^\infty e^{-sa}\mU(a)\,da.
$$
Multiplying \eqref{eq:upU} by $e^{-sa}\mEps(a\vz)$, integrating over $a$ and reversing the order of integration in the second term gives
$$
\mId=\mDelta(s-\vz)\widetilde\mU(s)
+\mDelta(\vgamma)\sum_{i=1}^m\mE_i\mQ\widetilde\mU(s/p_i).
$$
Equivalently,
\begin{equation}\label{eq:uprec}
\widetilde\mU(s)=\mA(s)\mB+\mA(s)\sum_{i=1}^m\mE_i\mQ\widetilde\mU(s/p_i),
\end{equation}
with
$$
\mA(s)=\mDelta\left(\frac{\vgamma}{\vz-s}\right),\qquad \mB=-\mDelta^{-1}(\vgamma).
$$

The form of~\eqref{eq:uprec} determines the appropriate indexing of its iterates. Each application of the second term selects a phase $i$, adds the transition $\mE_i\mQ$, and replaces the Laplace argument $s$ by $s/p_i$. After successive choices $i_1,\ldots,i_n$, the argument is $s/(p_{i_1}\cdots p_{i_n})$, so an $n$-fold iterate is naturally indexed by the corresponding phase word. Let $E^*=\bigcup_{n\ge0}E^n$ be the set of finite phase words, including the empty word $\varnothing$. For $\alpha=(i_1,\ldots,i_n)\in E^n$ define the prefix $\alpha\upord_k=(i_1,\ldots,i_k)$, $0\le k\le n$, with $\alpha\upord_0=\varnothing$, and set $p_\varnothing=1$ and $p_\alpha=\prod_{\ell=1}^{|\alpha|}p_{i_\ell}$.
For a nonempty phase word $\alpha=(i_1,\ldots,i_n)$ define
\begin{equation}\label{eq:Kword}
\mK_\alpha(s)=\mA(s)\mE_{i_1}\mQ\prod_{k=2}^{|\alpha|}\mA(s/p_{\alpha\upord_{k-1}})\mE_{i_k}\mQ,
\end{equation}
where the product is ordered from left to right in increasing $k$, with the empty product interpreted as $\mId$.
For a matrix-valued function $\mG$ define the dilation--transition operator and the one-step spatial operator by
\begin{equation}\label{eq:spatialoperators}
(\cS\mG)(s)=\sum_{i\in E}\mE_i\mQ\mG(s/p_i),
\qquad
(\cK\mG)(s)=\mA(s)(\cS\mG)(s).
\end{equation}

The spatial resolvent is
\begin{equation}\label{eq:Doperator}
\cR=(I-\cK)^{-1}=\sum_{n\ge0}\cK^n,
\qquad
(\cR \mF)(s)=\mF(s)+\sum_{n\ge1}\sum_{|\alpha|=n}\mK_\alpha(s)\mF(s/p_\alpha),
\end{equation}
whenever the series is justified. Here the phase words $\alpha$ arise naturally from expanding the powers of $\cK$: each application of $\cK$ selects one phase, so that $\cK^n$ is indexed by words of length~$n$.

The possible poles generated by the iterated factors $\mA(s/p_\alpha)$ occur at the countable set
$$
\cP=\{z_jp_\alpha:j\in E,\ \alpha\in E^*\}\subset(0,\bar z].
$$
\begin{proposition}[Green resolution of the spatial recursion]\label{prop:green.resolution}
Let $\mF$ be bounded on $(0,\infty)$.
For every real $s>0$ with $s\notin\cP$, the series defining $(\cR\mF)(s)$ converges absolutely and
\begin{equation}\label{eq:generic.green}
\mG(s)=(\cR(\mA\mF))(s)
\end{equation}
solves the recursion
\begin{equation}\label{eq:generic.recursion}
\mG=\mA\mF+\cK\mG,
\end{equation}
equivalently
$$\mG(s)=\mA(s)\mF(s)+\mA(s)\sum_{i=1}^m\mE_i\mQ\mG(s/p_i).$$
Conversely, any solution of~\eqref{eq:generic.recursion} such that $|\mG(s)|=O(s^{-1})$ as $s\to\infty$ is given by~\eqref{eq:generic.green}.
\end{proposition}

\begin{proof}
Fix $s>0$ with $s\notin\cP$, and choose $N_0$ so that $s/\bar p^{N_0}>2 \bar z$.
The finitely many word factors occurring during the first $N_0$ levels are finite.
For every subsequent factor, writing $r=p_{i_1}\cdots p_{i_k}$, $|\mA(s/r)|\le 2 \, \bar \gamma \, r/s$. Since $p_{\alpha\upord_k}\le{\bar p}^{\,k}$, the factor at depth $k+1$ is therefore bounded by a constant times ${\bar p}^{\,k}$.
Since $|\mE_i\mQ|\le1$, there is a constant $C_s<\infty$ such that, for every word $\alpha=(i_1,\ldots,i_n)$ with $n>N_0$,
$$
|\mK_\alpha(s)|
\le C_s\left(\frac{2 \bar \gamma}{s}\right)^{n-N_0}
 \bar p^{N_0+\cdots+(n-1)}.
$$
There are $m^n$ words of length $n$, so the super-geometric factor on the right proves absolute convergence of~\eqref{eq:Doperator}.
Separating the empty word from the series and grouping the remaining words according to their first phase gives~\eqref{eq:generic.recursion}, so~\eqref{eq:generic.green} is a solution.

Conversely, iterating~\eqref{eq:generic.recursion} $N$ times gives the corresponding partial sum of~\eqref{eq:generic.green} plus a remainder containing $\mG(s/p_\alpha)$ at words of length $N$.
The coefficient of each remainder term is $\mK_\alpha(s)$.
Since $|\mG(s)|=O(s^{-1})$ as $s\to\infty$, the additional factor $p_\alpha$ coming from $\mG(s/p_\alpha)$ gives the same super-geometric decay and sends the remainder to zero.
Thus~\eqref{eq:generic.green} is the unique solution in the stated class.
\end{proof}

\begin{corollary}[Upward Green series]\label{prop:upward.green}
For every real $s>\max\{\bar z,\kappa\}$,
\begin{equation}\label{eq:upgreen}
\widetilde\mU(s)=(\cR(\mA\mB))(s).
\end{equation}
\end{corollary}

\begin{proof}
Since $\mB$ is constant, $(\cR\mB)(s)=(\cR\mId)(s)\mB$.
Equation~\eqref{eq:uprec} has the form~\eqref{eq:generic.recursion} with $\mF=\mB$.
Since $\mU$ is continuous at $0$ with $\mU(0)=\mId$, its Laplace transform satisfies $|\widetilde\mU(s)|=O(s^{-1})$ as~$s\to\infty$, so Proposition~\ref{prop:green.resolution} gives~\eqref{eq:upgreen}.
\end{proof}

\begin{remark}\rm
The representation through $\cR(\mA\mB)$ is useful computationally because the truncation error decays faster than geometrically in the word length. 
In the scalar case, the corresponding expansion involves Pochhammer factors built from the geometric sequence $p^k$; see~\cite[Theorem 2.1]{vanderhofstad2023}. Phase modulation replaces this single geometric sequence by a tree of phase words, with the same rapid decay arising from the products $p_\alpha$.
\end{remark}

\section{The downward one-sided problem}\label{sec:downward}
For fixed $b>0$, we now consider the one-sided discounted passage matrix $\mZ(\D[x]{b})$ defined in~\eqref{Z}.
By definition, $\mZ(\D[x]{b})=\mId$ for $x\le b$.
Let $T$ be the first jump epoch.
Starting from $(x,i)$ with $x>b$, the process cannot reach $b$ before the first jump, since its level increases between jumps.
Conditioning on $T$ and on the post-jump phase gives
\begin{equation}\label{eq:downfirst}
\begin{aligned}
\mZ_{ij}(\D[x]{b})
&=\sum_{k=1}^m q_{ik}\int_0^\infty
\lambda_i e^{-(w+\lambda_i)t}
\mZ_{kj}\left(\D[p_i\left(x+\frac{t}{c_i}\right)]{b}\right) dt\\
&=\sum_{k=1}^m q_{ik}\gamma_i e^{z_ix}
\int_{p_ix}^\infty e^{-(z_i/p_i)y}\mZ_{kj}(\D[y]{b})\,dy.
\end{aligned}
\end{equation}
Define the first-jump operator~$\cD_b$ on $\cB_b$ by
\begin{equation}\label{eq:downoperator}
(\cD_b\mG)_{ij}(x)=
\begin{cases}
\displaystyle
\sum_{k=1}^m q_{ik}\int_0^\infty
\lambda_i e^{-(w+\lambda_i)t}
\mG_{kj}\left(p_i\left(x+\frac{t}{c_i}\right)\right)dt,
&x>b,\\
\delta_{ij} ,&x\le b.
\end{cases}
\end{equation}
Then~\eqref{eq:downfirst} is the fixed-point equation $\mZ(\D[\cdot]{b})=\cD_b\mZ(\D[\cdot]{b})$. Moreover, for $\mG_1,\mG_2\in\cB_b$,
$$
\|\cD_b\mG_1-\cD_b\mG_2\|_\infty
\le \max_{i\in E}\frac{\lambda_i}{w+\lambda_i}
\|\mG_1-\mG_2\|_\infty<\|\mG_1-\mG_2\|_\infty.
$$
Hence $\cD_b$ is a strict contraction, and the downward passage matrix is its unique bounded fixed point. In particular, the Picard iteration $\mG^{(n+1)}=\cD_b\mG^{(n)}$ converges uniformly to $\mZ(\D[\cdot]{b})$ for every $\mG^{(0)}\in\cB_b$.

Equivalently, rowwise,
\begin{equation}\label{eq:downrow}
\ve_i\mZ(\D[x]{b})=\gamma_i e^{z_ix}\ve_i\mQ
\int_{p_ix}^\infty e^{-(z_i/p_i)y}\mZ(\D[y]{b})\,dy,
\qquad x>b.
\end{equation}
We now transform this fixed-point equation in order to relate $\cD_b$ to the spatial operator $\cK$ of~Section~\ref{sec:upward}. Because~\eqref{eq:downrow} is valid only for $x>b$, the spatial transform must be taken on that same domain. For a bounded matrix-valued function $\mG$, set
$$
(\cEll_b\mG)(s)=\int_b^\infty e^{-sx}\mG(x)\,dx,\qquad s>0.
$$
Thus
\begin{equation}\label{eq:downtransform}
\widetilde\mZ_b(s)=(\cEll_b\mZ(\D[\cdot]{b}))(s).
\end{equation}
Since $0\le\mZ(\D[x]{b})\vi\le\vi$, the integral is finite for every $s>0$.

Applying $\cEll_b$ to~\eqref{eq:downrow} yields, rowwise,
$$
\ve_i\widetilde\mZ_b(s)
=\gamma_i\ve_i\mQ\int_b^\infty e^{(z_i-s)x}
\int_{p_ix}^\infty e^{-(z_i/p_i)y}\mZ(\D[y]{b})\,dy\,dx.
$$
The integration region in row $i$ is $x\ge b$ and $y\ge p_i x$. After reversing the order of integration and collecting the rows, we obtain
\begin{equation}\label{eq:downreverse}
\widetilde\mZ_b(s)
=\mA(s)\int_0^\infty \mDelta(y\ge \vp b)
\left[\mEps(-sy/\vp)-\mEps\bigl(b(\vz-s\vi)-(\vz/\vp)y\bigr)\right]
\mQ\mZ(\D[y]{b})\,dy.
\end{equation}
The post-jump variable enters the phase-dependent intervals $[p_i b,b]$, where $\mZ(\D[y]{b})=\mId$ is already prescribed. Separating this known contribution from the part over $[b,\infty)$ gives the boundary forcing
\begin{equation}\label{eq:Ri}
\begin{aligned}
\mR_b(s)
&=\int_0^b \mDelta(y\ge \vp b)
\left[\mEps(-sy/\vp)-\mEps\bigl(b(\vz-s\vi)-(\vz/\vp)y\bigr)\right]dy\\
&=\frac{1}{s}\mDelta(\vp)\left(e^{-sb}\mId-\mEps(-sb/\vp)\right)
-\mDelta(\vp/\vz)\mEps\bigl(b(\vz-s\vi)\bigr)
\left(\mEps(-b\vz)-\mEps(-b\vz/\vp)\right).
\end{aligned}
\end{equation}
Thus $\mR_b(s)$ is known explicitly in terms of the model parameters. 

The remaining part of~\eqref{eq:downreverse} gives
\begin{equation}\label{eq:downrowtransform}
\widetilde\mZ_b(s)
= \mA(s)\left(\mR_b(s)\mQ+e^{-bs}\mH_b\right) + (\cK\widetilde\mZ_b)(s),
\end{equation}
where
$$
\mH_b=-\sum_{i=1}^m e^{z_i b}\mE_i\mQ\widetilde\mZ_b(z_i/p_i).
$$
Thus, once $\mH_b$ is determined,~\eqref{eq:downrowtransform} has the same spatial recursion as the upward problem.

It remains to determine $\mH_b$ consistently with the finitely many transform values $\widetilde\mZ_b(z_i/p_i)$. For a provisional matrix $\mH$, define
\begin{equation}\label{eq:FbH}
\mF_b^{\mH}(s)=\mR_b(s)\mQ+e^{-bs}\mH.
\end{equation}
The corresponding transform recursion is
\begin{equation}\label{eq:downrecH}
\widetilde\mZ_b^{\mH}(s)
=\mA(s)\mF_b^{\mH}(s)
+\mA(s)\sum_{i=1}^m\mE_i\mQ\widetilde\mZ_b^{\mH}(s/p_i).
\end{equation}

For $\mH=\mH_b$, equation~\eqref{eq:downrecH} is exactly the Laplace transform of the unique fixed-point equation for $\cD_b$. To recover $\mH_b$, we impose its defining relation on the transform $\widetilde\mZ_b^{\mH}$. If $z_i/p_i\notin\cP$ for every $i\in E$, define the affine map
\begin{equation}\label{eq:Cbmap}
\cC_b(\mH)=-\sum_{i=1}^m e^{z_i b}\mE_i\mQ\widetilde\mZ_b^{\mH}(z_i/p_i).
\end{equation}
By the definition of $\mH_b$, the compatibility condition is
\begin{equation}\label{eq:Hbfix}
\mH_b=\cC_b(\mH_b),
\end{equation}
and the downward transform is $\widetilde\mZ_b=\widetilde\mZ_b^{\mH_b}$. 

\begin{theorem}[Downward Green representation]\label{thm:downward}
The one-sided downward transform satisfies
\begin{equation}\label{eq:downgreen}
\widetilde\mZ_b(s)=\bigl(\cR(\mA\mF_b^{\mH_b})\bigr)(s),
\end{equation}
for every real $s>0$ with $s\notin\cP$, where $\mH_b$ satisfies the compatibility condition~\eqref{eq:Hbfix}. The series is absolutely convergent. At $s\in\cP$, the apparent singularities on the right-hand side are removable, and~\eqref{eq:downgreen} extends by continuity.
\end{theorem}

\begin{proof}
Taking $\mH=\mH_b$ in~\eqref{eq:downrecH} gives~\eqref{eq:downrowtransform}. The known intervals contribute $\mA(s)\mR_b(s)\mQ$, while the last sum in~\eqref{eq:downrowtransform} is $\mA(s)e^{-bs}\mH_b$. Thus the transform equation has the form~\eqref{eq:generic.recursion}.
Since $\mF_b^{\mH_b}$ is bounded on $(0,\infty)$ and, because $0\le\mZ(\D[x]{b})\vi\le\vi$, $|\widetilde\mZ_b(s)|\le e^{-bs}/s$ for every real $s>0$, Proposition~\ref{prop:green.resolution} gives~\eqref{eq:downgreen} and absolute convergence for $s\notin\cP$. Moreover, $\widetilde\mZ_b$ is continuous on~$(0,\infty)$ as the Laplace transform of a bounded function on $[b,\infty)$. Hence, at every $s\in\cP$, the complete expression on the right-hand side of~\eqref{eq:downgreen} has the finite limit $\widetilde\mZ_b(s)$ as $r\to s$, $r>0$, $r\notin\cP$.
\end{proof}

\begin{remark}[Scalar check]\label{rem:scalarcheck}
\rm When $m=1$, $c_1=1$, $Q=(1)$, $p_1=p$ and $\lambda_1=\lambda$, equation~\eqref{eq:downrowtransform} becomes
$$
\widetilde Z(\D[s]{b})
=\frac{\lambda}{p(w+\lambda-s)}
\left(R_b(s)+\widetilde Z(\D[s/p]{b})
-e^{(w+\lambda-s)b}\widetilde Z(\D[(w+\lambda)/p]{b})\right),
$$
with $R_b(s)$ given by the scalar version of~\eqref{eq:Ri}. After elementary simplification of the explicit~$R_b(s)$ term, this is exactly the scalar downward transform recursion in~\cite{vanderhofstad2023}.
\end{remark}

\begin{remark}\rm
The dependence of $\widetilde\mZ_b^{\mH}$ on $\mH$ can be affine, so the compatibility condition~\eqref{eq:Hbfix} reduces to finite-dimensional linear algebra.
With $\mH_0(s)=\mA(s)\mR_b(s)\mQ$ and $\mH_1(s)=\mA(s)e^{-bs}$, define the matrices
$$
\mGamma_b=\sum_{i=1}^m e^{z_i b}\mE_i\mQ\,\cR\mH_0(z_i/p_i)
,\qquad
\mPi_b=\sum_{i=1}^m e^{z_i b}\mE_i\mQ\,\cR\mH_1(z_i/p_i).
$$
By linearity of $\cR$, we have $\cC_b(\mH)=-\mGamma_b-\mPi_b\,\mH$, and the fixed-point equation~\eqref{eq:Hbfix} is equivalent to
\begin{equation}\label{eq:finiteclosure}
(\mId+\mPi_b)\mH_b=-\mGamma_b.
\end{equation}
If $\mId+\mPi_b$ is nonsingular, this determines $\mH_b$ and hence~\eqref{eq:downgreen} by finite-dimensional linear algebra.
\end{remark}


\section{Two-sided exit}\label{sec:two-sided}
For $0\le b<a$ define
\begin{equation}\label{eq:upward.twosided}
\mZ_{ij}(w;\UD[x]{a}{b})
=\bE_{x,i}[e^{-w\taua};\taua<\taub,J(\taua)=j].
\end{equation}
We set $\mZ(\UD[x]{a}{b})=0$ for $x\le b$ and $\mZ(\UD[x]{a}{b})=\mId$ for $x\ge a$.
With $\bar p=\max_i p_i$, the recursion below needs an initial strip, which is explicit.

\begin{proposition}[Initial two-sided strip]\label{prop:L.seed}
Assume $b>0$ and $b<a\le b/\bar p$. Then, for $b<x\le a$,
\begin{equation}\label{eq:Lseed}
\mZ(\UD[x]{a}{b})=\mEps(-(a-x)\vz).
\end{equation}
If $b=0$, then $\taub=\infty$ almost surely for every positive starting level and therefore
$$
\mZ(\UD[x]{a}{0})=\mZ(\U[x]{a}),\qquad 0<x\le a.
$$
\end{proposition}

\begin{proof}
Suppose first that $b>0$. Before the process reaches $a$, its pre-jump level is strictly smaller than $a$. Hence a jump from phase $i$ lands below or at $p_i a\le \bar pa\le b$.
Consequently the event $\{\taua<\taub\}$ is exactly the event that no jump occurs during the deterministic climb from~$x$ to $a$. Starting in phase $i$, that climb takes time $c_i(a-x)$, the phase does not change, and the discounted probability of no jump is $e^{-z_i(a-x)}$.
This gives \eqref{eq:Lseed}. If $b=0$, positivity of the drift and the multiplicative factors $p_i\in(0,1)$ imply that a process started from $x>0$ remains strictly positive at every finite time, so the lower boundary is never attained and the two-sided upward transform reduces to the one-sided one.
\end{proof}

\begin{proposition}[Cocycle property]\label{prop:L.cocycle}
If $b<x\le a'<a$, then
\begin{equation}\label{eq:Lcocycle}
\mZ(\UD[x]{a}{b})
=\mZ(\UD[x]{a'}{b})\mZ(\UD[a']{a}{b}).
\end{equation}
Moreover $\mZ(\UD[x]{a}{b})$ is invertible for $b<x\le a$.
\end{proposition}

\begin{proof}
On $\{\taua<\taub\}$ the continuous upward path must first reach $a'$. Applying the strong Markov property there and summing over the phase at $a'$ gives \eqref{eq:Lcocycle}. For $h\downarrow0$, the matrix~$\mZ(\UD[x]{x+h}{b})$ converges to $\mId$ whenever $x>b$,
 and we can use similar arguments like in the proof of Theorem \ref{thm:upward} to show
 invertibility of $\mZ(\UD[x]{a}{b})$.
\end{proof}

The cocycle gives a constructive continuation in the upper level.
Suppose that $\mZ(\UD[y]{a}{b})$ is known for every $y<a$. If $x\in[a,a/\bar p)$, then $p_i x<p_i a/\bar p\le a$ for every $i$, so a jump occurring before $a/\bar p$ lands in a region where the old two-sided transform is already known; see Figure~\ref{fig:continuation-layer}.

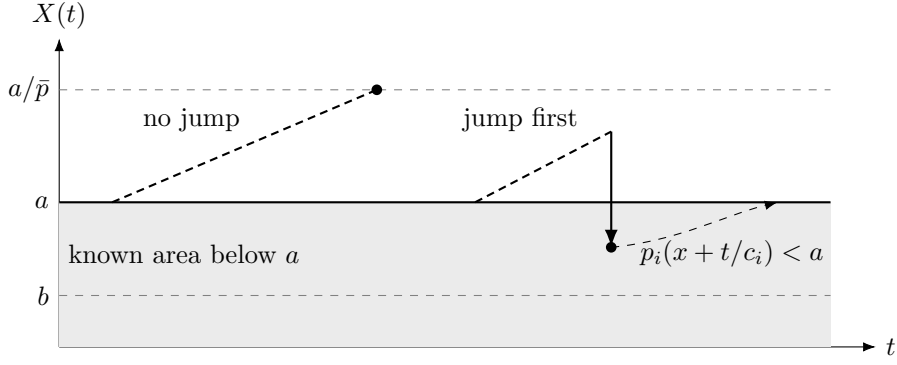
\begin{figure}[t]
\centering
\begin{tikzpicture}[x=1cm,y=0.85cm,>=Latex,font=\small]
\draw[->] (0,0)--(10.8,0) node[right] {$t$};
\draw[->] (0,0)--(0,4.8) node[above] {$X(t)$};
\fill[black!8] (0,0) rectangle (10.2,2.25);
\draw[dashed,gray] (0,0.8)--(10.2,0.8); \node[left] at (0,0.8) {$b$};
\draw[thick] (0,2.25)--(10.2,2.25); \node[left] at (0,2.25) {$a$};
\draw[dashed,gray] (0,4.0)--(10.2,4.0); \node[left] at (0,4.0) {$a/\bar p$};
\draw[black,thick,densely dashed] (0.7,2.25)--(4.2,4.0);
\fill[black] (4.2,4.0) circle (2pt);
\node[above] at (1.75,3.2) {no jump};
\draw[black,thick,densely dashed] (5.5,2.25)--(7.3,3.35);
\draw[black,thick,->] (7.3,3.35)--(7.3,1.55);
\fill[black] (7.3,1.55) circle (2pt);
\draw[black,dashed,->] (7.3,1.55)..controls(8.2,1.65)and(8.8,2.05)..(9.5,2.25);
\node[right] at (7.55,1.45) {$p_i(x+t/c_i)<a$};
\node[above] at (6.1,3.2) {jump first};
\node[black,anchor=north] at (1.65,1.75) {known area below $a$};
\end{tikzpicture}
\caption{One continuation layer for the two-sided transform. From $a$ to $a/\bar p$, either the upper level is reached without a jump, or a jump lands below $a$, where the transform is already known.}
\label{fig:continuation-layer}
\end{figure}

For $x\in[a,a/\bar p)$ define
\begin{equation}\label{eq:Pdef}
\mP_{ij}(x)=\gamma_i e^{z_i x}
\int_{p_ix}^{p_i a/\bar p}
\sum_{h=1}^m q_{ih}\mZ_{hj}(\UD[y]{a}{b})
e^{-(z_i/p_i)y}\,dy.
\end{equation}

\begin{theorem}[Recursive upper continuation]\label{thm:two.sided.up}
For $x\in[a,a/\bar p)$,
\begin{equation}\label{eq:Lcontinuation}
\mZ(\UD[x]{a/\bar p}{b})
=\mEps(-(a/\bar p-x)\vz)
+\mP(x)\mZ(\UD[a]{a/\bar p}{b}),
\end{equation}
where
\begin{equation}\label{eq:Lendpoint}
\mZ(\UD[a]{a/\bar p}{b})
=(\mId-\mP(a))^{-1}\mEps(-(a/\bar p-a)\vz).
\end{equation}
In particular, knowledge of the upward two-sided transform below $a$ determines it below $a/\bar p$.
\end{theorem}

\begin{proof}
Fix the initial phase $i$ and put $d_i(x)=c_i(a/\bar p-x)$. If the first jump time $T_i$ exceeds~$d_i(x)$, the process reaches $a/\bar p$ continuously with phase $i$, contributing $e^{-(w+\lambda_i)d_i(x)}\delta_{ij}=e^{-z_i(a/\bar p-x)}\delta_{ij}$.
If $T_i=t<d_i(x)$, the post-jump state is $p_i(x+t/c_i)<a$. The cocycle property therefore gives
$$
\mZ\left(\UD[p_i(x+t/c_i)]{a/\bar p}{b}\right)
=\mZ\left(\UD[p_i(x+t/c_i)]{a}{b}\right)
\mZ(\UD[a]{a/\bar p}{b}).
$$
Conditioning on $T_i$ and the post-jump phase yields
$$
\begin{aligned}
\mZ_{ij}(\UD[x]{a/\bar p}{b})
={}&e^{-z_i(a/\bar p-x)}\delta_{ij}\\
&+\sum_{h,k}q_{ih}\int_0^{d_i(x)}
\lambda_i e^{-(w+\lambda_i)t}
\mZ_{hk}\left(\UD[p_i(x+t/c_i)]{a}{b}\right)dt\,
\mZ_{kj}(\UD[a]{a/\bar p}{b}).
\end{aligned}
$$
With $y=p_i(x+t/c_i)$, the jump contribution is $\mP(x)\mZ(\UD[a]{a/\bar p}{b})$, which gives \eqref{eq:Lcontinuation}. Evaluating this identity at the interface $x=a$ gives $(\mId-\mP(a))\mZ(\UD[a]{a/\bar p}{b})=\mEps(-(a/\bar p-a)\vz)$.
Finally, since the row sums of $\mZ(\UD[y]{a}{b})$ are at most one,
$$
|\mP(a)|
\le\max_i\left\{\int_0^{c_i(a/\bar p-a)}\lambda_i e^{-(w+\lambda_i)t}\,dt\right\}
=\max_i\left\{\frac{\lambda_i}{w+\lambda_i}\left(1-e^{-z_i(a/\bar p-a)}\right)\right\}<1.
$$
Hence $\mId-\mP(a)$ is invertible and $(\mId-\mP(a))^{-1}=\sum_{n\ge0}\mP(a)^n$.
\end{proof}

Proposition~\ref{prop:L.seed} and Theorem~\ref{thm:two.sided.up} give a complete constructive scheme. For $b>0$ start with any upper level $a_0\le b/\bar p$, where the transform is diagonal and explicit. Repeatedly replace $a$ by $a/\bar p$; at each step the integral \eqref{eq:Pdef} uses only values already computed below the previous upper level. For $b=0$ the initial data are supplied by the one-sided upward transform of Theorem~\ref{thm:upward}. Thus no phase-dependent vector barrier is needed for the scalar upper and lower levels considered here.

Define the downward two-sided transform by
$$
\mZ_{ij}(w;\DU[x]{a}{b})
=\bE_{x,i}[e^{-w\taub};\taub<\taua,J(\taub)=j].
$$

\begin{corollary}[Downward two-sided exit]\label{cor:two.sided.down}
For $b\le x<a$,
\begin{equation}\label{eq:twosideddown}
\mZ(\DU[x]{a}{b})
=\mZ(\D[x]{b})-\mZ(\UD[x]{a}{b})\mZ(\D[a]{b}).
\end{equation}
\end{corollary}

\begin{proof}
Decompose the one-sided downward passage according to whether $b$ is reached before $a$. On $\{\taua<\taub\}$ apply the strong Markov property at $\taua$, where the level is exactly $a$. This gives
$$
\mZ(\D[x]{b})=\mZ(\DU[x]{a}{b})
+\mZ(\UD[x]{a}{b})\mZ(\D[a]{b}),
$$
which is \eqref{eq:twosideddown}.
\end{proof}

\section{Reflection at the running infimum}\label{sec:reflected}
We finally turn to first passage for the process reflected at its running infimum.
Starting from an initial running minimum $v>0$, define
$$
\uX_t=v\wedge\inf_{0\le s\le t}X(s),\qquad Y_t=X(t)-\uX_t,
\qquad \tau_a=\inf\{t\ge0:Y_t\ge a\}.
$$
Fix $w,a>0$.
For $b>0$ let $\rhob=\inf\{t>0:X(t)<b\}$; the strict convention in $t$ is needed when the process starts from the current minimum.
To underline the dependence on $\underline{X}_0=v$ we introduce notation $\mathbb{E}_{(x,v),i}[\cdot]=\mathbb{E}[\cdot| X_0=x, \underline{X}_0=v, J_0=i]$.
We define
\begin{equation}\label{eq:ref.H}
(\mH_a(v))_{ij}:=\bE_{(v,v),i}[e^{-w\tau_a};\tau_a<\infty,J(\tau_a)=j]
\end{equation}
Observe that if a jump from $(x,v)$ in phase $i$ satisfies $p_i x<v$, then the post-jump state is $(p_i x,p_i x)$ with a new phase $k$ chosen according to $q_{ik}$.
Hence the reflected coordinate returns to zero, but the continuation matrix is $\mH_a(p_i x)$ with the post-jump phase retained.

For $v<x<v+a$ we use  
$$
\mZ_{ij}(w;\UD[x]{v+a}{v})
=\bE_{(x,v),i}[e^{-w\tau_{\U{v+a}}};\tau_{\U{v+a}}<\rhov,J(\tau_{\U{v+a}})=j].
$$
and use the entrance notation
$$
\mZ_{ij}(w;\UD[v^+]{v+a}{v})
=\bE_{(v,v),i}[e^{-w\tau_{\U{v+a}}};\tau_{\U{v+a}}<\rhov,J(\tau_{\U{v+a}})=j].
$$
For a bounded Borel matrix function $\mG:(0,\infty)\to\bR^{m\times m}$ define the restart operator
\begin{equation}\label{eq:ref.S}
(\cV_a\mG)_{ij}(v):=\bE_{(v,v),i}[e^{-w\rhov}\mG_{J(\rhov),j}(X(\rhov));\rhov<\tau_{\U{v+a}}],
\end{equation}
and the affine Picard map
\begin{equation}\label{eq:ref.T}
(\cT_a\mG)(v):= \mZ(\UD[v^+]{v+a}{v}) + (\cV_a\mG)(v).
\end{equation}

\begin{lemma}[Strict contraction]\label{lem:ref.contraction}
The map $\cT_a$ is a strict contraction on bounded Borel matrix functions with the supremum row-sum norm, with Lipschitz constant
\begin{equation}\label{eq:ref.beta}
\beta:=\max_{i\in E}\frac{\lambda_i}{w+\lambda_i}\left(1-e^{-z_i a}\right)<1.
\end{equation}
\end{lemma}

\begin{proof}
Let $\mG_1$ and $\mG_2$ be bounded Borel matrix functions.
Fix $v>0$ and $i\in E$.
On ${\rhov<\tau_{\U{v+a}}}$, the first jump occurs before the no-jump climb from $v$ to $v+a$, which takes time $c_i a$ when the initial phase is $i$.
Thus, if $T_i\sim\Exp(\lambda_i)$ denotes the first jump time, then $T_i<c_i a$ and $\rhov\ge T_i$ on this event.
Hence,
$$
\bE_{(v,v),i}\left[e^{-w\rhov};,\rhov<\tau_{\U{v+a}}\right]
\le \bE\left[e^{-wT_i};,T_i<c_i a\right]
=\frac{\lambda_i}{w+\lambda_i}\left(1-e^{-z_i a}\right)
\le\beta.
$$
Since the entrance term in \eqref{eq:ref.T} does not depend on $\mG$, the definition of $\cV_a$ and the preceding estimate give
$$
\|\cT_a\mG_1-\cT_a\mG_2\|_\infty
\le\beta\|\mG_1-\mG_2\|_\infty.
$$
\end{proof}

\begin{theorem}[Diagonal restart value]\label{thm:ref.restart}
The matrix-valued function $H_a$ defined in \eqref{eq:ref.H} is the unique bounded Borel function satisfying
\begin{equation}\label{eq:ref.Hfixed}
\mH_a=\cT_a\mH_a.
\end{equation}
Moreover, we have $0<\mH_a(v)\vi<\vi$ componentwise for every $v>0$.
If $\mH_a^{(0)}=0$ and
\begin{equation}\label{eq:ref.picard}
\mH_a^{(n+1)}=\cT_a\mH_a^{(n)},
\end{equation}
then $\|\mH_a-\mH_a^{(n)}\|_\infty\le\beta^n$.
\end{theorem}

\begin{proof}
By the strong Markov property at $\rhov\wedge\tau_{\U{v+a}}$, the matrix-valued function $\mH_a$ defined in~\eqref{eq:ref.H} satisfies \eqref{eq:ref.Hfixed}.
Its uniqueness follows from Lemma~\ref{lem:ref.contraction}.
The inequality $\mH_a(v)\mathbf{1}>0$ follows from the event that no jump occurs before the process reaches $v+a$, while $\mH_a(v)\mathbf{1}<\mathbf{1}$ follows from $\tau_a>0$ almost surely and $w>0$.
Finally, since $\mH_a$ is a fixed point of the contraction $\cT_a$ and $|\mH_a|\infty\le1$,
$$
\|\mH_a-\mH_a^{(n)}\|_\infty
\le\beta^n\|\mH_a-\mH_a^{(0)}\|_\infty
\le\beta^n,
$$
which proves the claimed estimate.
\end{proof}

The Picard iteration \eqref{eq:ref.picard} acts on whole matrix functions of $v$.
We now realize one application of $\cT_a$ by a finite method of steps.
For $k\ge0$, set
$$
x_k(v):=\frac{v}{\bar p^k},\qquad \cD_k:=\{(x,v):v>0,\ x_k(v)\le x\le x_{k+1}(v)\};
$$
see Figure~\ref{fig:geometric-regions}.
If $(x,v)\in\cD_k$ and $p_i x\ge v$, then $p_i x\le\bar p x\le x_k(v)$, so every post-jump argument lies either below the diagonal or in a region already constructed; we proceed from left to right.

\begin{figure}[t]
\centering
\begin{tikzpicture}[x=0.85cm,y=0.85cm,>=Latex,font=\small]
\draw[->] (0,0)--(8.4,0) node[right] {$x$};
\draw[->] (0,0)--(0,6.4) node[above] {$v$};
\fill[gray!8] (0,0)--(7.7,7.7)--(7.7,4.62)--cycle;
\fill[gray!16] (0,0)--(7.7,4.62)--(7.7,2.77)--cycle;
\fill[gray!24] (0,0)--(7.7,2.77)--(7.7,1.66)--cycle;
\draw[thick] (0,0)--(6.2,6.2) node[above right] {$v=x$};
\draw[thick] (0,0)--(8.0,4.8) node[right] {$v=\bar p x$};
\draw[thick] (0,0)--(8.0,2.88) node[right] {$v=\bar p^2x$};
\draw[thick] (0,0)--(8.0,1.73) node[right] {$v=\bar p^3x$};
\node at (7.1,5.25) {$\cD_0$};
\node at (7.1,3.25) {$\cD_1$};
\node at (7.1,2.05) {$\cD_2$};
\draw[->,black,thick] (6.3,2.35)--(4,2.35);
\node[black,above] at (5.55,2.35) {$x\mapsto p_i x$};
\end{tikzpicture}
\caption{The $\bar p$-geometric regions $\cD_k$. A jump from $\cD_k$ lands in the already constructed part of the domain unless it creates a new minimum.}
\label{fig:geometric-regions}
\end{figure}
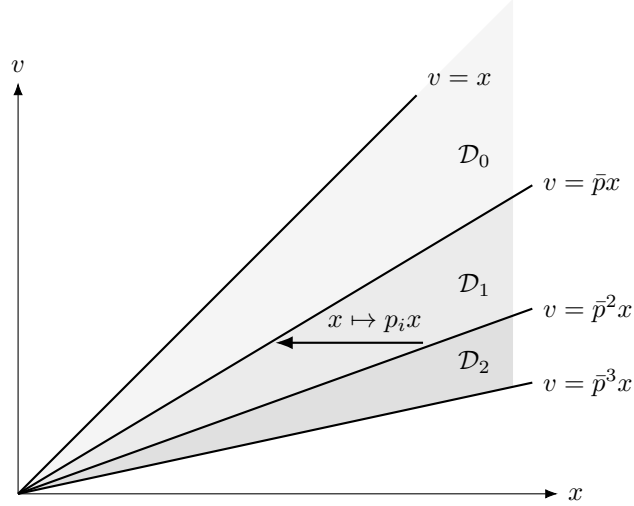

Given a bounded Borel $\mG$, define $\mK(x,v)$ and $(\cJ\mG)(x,v)$ recursively over the regions $\cD_k$ by~$\mK(v,v)=\mId$, $(\cJ\mG)(v,v)=0$, and, rowwise for Lebesgue-a.e. $x>v$,
\begin{align}\label{eq:ref.K}
\partial_x(\ve_i\mK(x,v))
={}&z_i\ve_i\mK(x,v)-c_i\lambda_i\indF{p_ix\ge v}\sum_{k=1}^m q_{ik}\ve_k\mK(p_ix,v), \\
\partial_x(\ve_i(\cJ\mG)(x,v))
={}&z_i\ve_i(\cJ\mG)(x,v)-c_i\lambda_i\indF{p_ix\ge v}\sum_{k=1}^m q_{ik}\ve_k(\cJ\mG)(p_ix,v)\label{eq:ref.J}\\
&+c_i\lambda_i\indF{p_ix<v}\sum_{k=1}^m q_{ik}\ve_k\mG(p_ix).\nonumber
\end{align}
The recursion is non-circular because every argument $(p_i x,v)$ above the diagonal lies in the already constructed part of the domain.
For $v>0$ put
\begin{equation}\label{eq:ref.kstar}
k_a^*(v):=\min\{k\ge0:v+a\le x_{k+1}(v)\}.
\end{equation}
Only $\cD_0,\ldots,\cD_{k_a^*(v)}$ are needed at $(v+a,v)$.

For $v<x<v+a$ let $\mF_{\mG}(\cdot,v)$ denote a 
bounded solution of the following equation
\begin{equation}\label{eq:ref.FG.harmonica1}
\frac1{c_i}\partial_x(\ve_i\mF_{\mG}(x,v))
-(w+\lambda_i)\ve_i\mF_{\mG}(x,v)
+\lambda_i\sum_kq_{ik}\ve_k\mG(p_ix)=0
\end{equation}
if $p_ix<v$ and
\begin{equation}\label{eq:ref.FG.harmonica2}
\frac1{c_i}\partial_x(\ve_i\mF_{\mG}(x,v))
-(w+\lambda_i)\ve_i\mF_{\mG}(x,v)
+\lambda_i\sum_kq_{ik}\ve_k
%
\mF_{\mG}(p_ix,v)
=0
\end{equation}
if $p_ix\ge v$.

\begin{proposition}[Global realization of the Picard map]\label{prop:ref.deterministic}
For every bounded Borel matrix function $\mG$, $\mK(v+a,v)$ is invertible and
\begin{equation}\label{eq:ref.T.det}
(\cT_a\mG)(v)=\mK(v+a,v)^{-1}\left(\mId+(\cJ\mG)(v+a,v)\right),\qquad v>0.
\end{equation}
Moreover, if $0\le\mG\vi\le\vi$, then $0\le\mF_{\mG}(x,v)\vi\le\vi$ for $v\le x\le v+a$, and
\begin{equation}\label{eq:ref.FG}
\mF_{\mG}(x,v)=\mK(x,v)(\cT_a\mG)(v)-(\cJ\mG)(x,v).
\end{equation}
\end{proposition}

\begin{proof}
Equations \eqref{eq:ref.K}--\eqref{eq:ref.J} show that every solution of \eqref{eq:ref.FG.harmonica1}-\eqref{eq:ref.FG.harmonica2} has the form
$$
\mF_{\mG}(x,v)=\mK(x,v)\mC-(\cJ\mG)(x,v)
$$
for a constant matrix $\mC$.
The method of steps gives uniqueness successively on the finitely many regions intersecting $[v,v+a]$.
For $\mG=0$, the probabilistically defined continuation problem has a solution with terminal value $\mId$ at $v+a$.
Hence its representation above yields~$\mK(v+a,v)\mC=\mId$ for some matrix $\mC$. Thus~$\mK(v+a,v)$ has a right inverse and, being square, is invertible.
For general $\mG$, imposing $\mF_{\mG}(v+a,v)=\mId$ yields
$$
\mC=\mK(v+a,v)^{-1}\left(\mId+(\cJ\mG)(v+a,v)\right).
$$
Since $\mK(v,v)=\mId$ and $(\cJ\mG)(v,v)=0$, the value at $(v,v)$ is $\mC$; by the first-step decomposition it is precisely $(\cT_a\mG)(v)$.
The same decomposition shows that $0\le\mF_{\mG}(x,v)\vi\le\vi$ whenever $0\le\mG\vi\le\vi$.
This proves \eqref{eq:ref.T.det} and \eqref{eq:ref.FG}.
\end{proof}

At the fixed point, define
\begin{equation}\label{eq:ref.Fa}
\mF_a=\mF_{\mH_a}.
\end{equation}
By Theorem~\ref{thm:ref.restart} and Proposition~\ref{prop:ref.deterministic}, $0\le\mF_a(x,v)\vi\le\vi$ and
$$
\mF_a(v,v)=\mH_a(v),\qquad \mF_a(v+a,v)=\mId.
$$
Then \eqref{eq:ref.FG.harmonica1}-\eqref{eq:ref.FG.harmonica2} becomes
\begin{equation}\label{eq:ref.Fa.harmonic}
\frac1{c_i}\partial_x(\ve_i\mF_a(x,v))-w\ve_i\mF_a(x,v)
+\lambda_i\sum_kq_{ik}\ve_k\mF_a(p_ix,\min\{v,p_ix\})-\lambda_i\ve_i\mF_a(x,v)=0,
\end{equation}
where $\mF_a(y,y)=\mH_a(y)$.

\begin{theorem}[Reflected first passage]\label{thm:reflected}
For $0<v\le x<v+a$,
\begin{equation}\label{eq:ref.answer}
(\mF_a(x,v))_{ij}=\bE_{(x,v),i}[e^{-w\tau_a};\tau_a<\infty,J(\tau_a)=j].
\end{equation}
In particular $\mF_a(v,v)=\mH_a(v)$.
\end{theorem}

\begin{proof}
The definition of $\mF_a$ is the first-step recursion for the reflected passage problem with continuation value $\mH_a(y)$ when a jump creates a new minimum $y$.
Indeed, on $\{\tau_{\U{v+a}}<\rhov\}$ passage is completed during the current excursion, while on $\{\rhov<\tau_{\U{v+a}}\}$ the process restarts from the diagonal state $(X(\rhov),X(\rhov))$ with phase $J(\rhov)$.
By \eqref{eq:ref.H}, the continuation transform from that state is $\mH_a(X(\rhov))$.
The strong Markov property at $\rhov$ therefore identifies the unique solution of the recursion with the right-hand side of \eqref{eq:ref.answer}.
\end{proof}

The martingale property is now an immediate consequence of the preceding identification.
\begin{proposition}[Doob martingale]\label{prop:ref.martingale}
For $v\le x<v+a$,
\begin{equation}\label{eq:ref.M}
M_t=e^{-w(t\wedge\tau_a)}\ve_{J(t\wedge\tau_a)}\mF_a(X(t\wedge\tau_a),\uX_{t\wedge\tau_a})
\end{equation}
is a bounded row-vector martingale.
\end{proposition}

\begin{proof}
Define the bounded terminal row vector $\Xi=e^{-w\tau_a}\ve_{J(\tau_a)}\indF{\tau_a<\infty}$, then, by Theorem~\ref{thm:reflected} and the strong Markov property at $t\wedge\tau_a$, we have that $M_t=\bE_{x,i}[\Xi\mid\cF_{t\wedge\tau_a}]$.
On $\{\tau_a\le t\}$, passage through $a$ is continuous from below.
Hence $X(\tau_a)=\uX_{\tau_a}+a$ and $\mF_a(X(\tau_a),\uX_{\tau_a})=\mId$.
On $\{t<\tau_a\}$, the identity is Theorem~\ref{thm:reflected} applied from $(X(t),\uX_t,J(t))$.
Thus $M$ is the Doob martingale of $\Xi$.
\end{proof}

\begin{remark}[Relation with the scalar proof of \cite{vanderhofstad2023}]
\label{rem:ref.scalar}\rm 
The martingale \eqref{eq:ref.M} appears in the proof of Theorem~3.4 of \cite{vanderhofstad2023}.
We note that an issue arises in that proof when the Dynkin-Lebesgue-Stieltjes formula is applied to $f(y,z)=F(w;y,a(w,c,z),z)$, since with the extension $F(w;y,a,z)=1$ for $y\le z$, the required differentiability with respect to $z$ does not hold along the diagonal $y=z$.
Consequently, the hypotheses of the formula are not satisfied at that point, so the subsequent martingale conclusion does not follow from this application.
Here we follow a different approach to the reflected first passage problem, which yields the correct transform, including in the scalar case $m=1$.
\end{remark}

The same construction gives the following deterministic approximation.

\begin{theorem}[Constructive approximation]\label{thm:ref.certified}
Fix a compact computational range $v\in[v_0,v_1]$ with $v_0>0$ and let $\beta$ be given by \eqref{eq:ref.beta}.
Then the transform in \eqref{eq:ref.answer} can be approximated by a deterministic sequence of finite quadratures with explicit error propagation. 
\begin{enumerate}[label=(\roman*)]
\item Start from $\mH_a^{(0)}=0$ and use \eqref{eq:ref.T.det} to compute $\mH_a^{(n+1)}=\cT_a\mH_a^{(n)}$.
One exact Picard step acts on the whole matrix function of $v$; at each fixed $v$ its evaluation uses only $\cD_0,\ldots,\cD_{k_a^*(v)}$.
\item The exact Picard error satisfies $\|\mH_a-\mH_a^{(n)}\|_\infty\le\beta^n$.
\item If the numerical realization of the $k$-th Picard step has certified sup-norm error at most $\delta_k$, including interpolation and quadrature errors, and $\widehat{\mH}_a^{(0)}=0$, then
\begin{equation}\label{eq:ref.num.error}
\|\widehat{\mH}_a^{(n)}-\mH_a\|_\infty\le\beta^n+\sum_{k=0}^{n-1}\beta^{n-1-k}\delta_k.
\end{equation}
\item Once $\widehat{\mH}_a$ is available, compute $\widehat{\mF}_a(x,u)$ from \eqref{eq:ref.FG} by the same finite method of steps.
If this final calculation has error at most $\eta$, then
$$
\|\widehat{\mF}_a(x,u)-\mF_a(x,u)\|\le\|\widehat{\mH}_a-\mH_a\|_\infty+\eta.
$$
\end{enumerate}
Thus the matrix Laplace transform is obtained by a deterministic sequence of finite quadratures with a prescribed total error.
\end{theorem}

\begin{proof}
The first two assertions are Proposition~\ref{prop:ref.deterministic} and Theorem~\ref{thm:ref.restart}.
If $e_n=\|\widehat{\mH}_a^{(n)}-\mH_a\|_\infty$, contraction gives $e_{n+1}\le\beta e_n+\delta_n$, which yields \eqref{eq:ref.num.error}.
Finally the continuation representation is affine in the restart payoff with a discounted subprobability kernel, so replacing $\mH_a$ by $\widehat{\mH}_a$ changes the exact continuation value by at most $\|\widehat{\mH}_a-\mH_a\|_\infty$; adding the final numerical error gives the last estimate.
\end{proof}

\begin{remark}[Implementation near zero]
\rm The exact Picard recursion is global and requires no boundary value at $v=0$.
Numerically, $k_a^*(v)\to\infty$ as $v\downarrow0$, so a finite implementation is performed on a prescribed positive range and with a finite representation of the $v$-dependence.
\end{remark}

\end{document}